\documentclass{amsart}

\usepackage{url}
\usepackage[cmtip,all]{xy}

\newtheorem{theorem}{Theorem}[section]
\newtheorem{lemma}[theorem]{Lemma}
\newtheorem{proposition}[theorem]{Proposition}

\theoremstyle{definition}
\newtheorem{definition}[theorem]{Definition}

\numberwithin{equation}{section}

\newcommand{\de}{\colon}
\newcommand{\st}{\mid}
\newcommand{\C}{\mathbb{C}}
\newcommand{\Q}{\mathbb{Q}} 
\newcommand{\Ar}{\mathcal{A}}
\DeclareMathOperator{\im}{im}
\DeclareMathOperator{\dist}{dist}
\DeclareMathOperator{\rk}{rk}
\DeclareMathOperator{\codim}{codim}
\DeclareMathOperator{\sgn}{sgn}

\begin{document}
\title[Subspace arrangements]{Rational homotopy type of subspace arrangements with a geometric lattice}

\author{Gery Debongnie}
\address{UCL, Departement de mathematique \\ Chemin du Cyclotron, 2 \\ B-1348 Louvain-la-neuve \\ Belgium}
\email{debongnie@math.ucl.ac.be}
\thanks{The author is an ``Aspirant'' of the ``Fonds National pour la Recherche Scientifique'' (FNRS), Belgium.}


\begin{abstract}
Let $\Ar = \{x_1, \dotsc, x_n\}$ be a subspace arrangement with a geometric lattice such that $\codim(x) \geq 2$ for every $x \in\Ar$. Using rational homotopy theory, we prove that the complement $M(\Ar)$ is rationally elliptic if and only if the sum $x_1^\perp + \dotso + x_n^\perp$ is a direct sum. The homotopy type of $M(\Ar)$ is also given~: it is a product of odd dimensional spheres. Finally, some other equivalent conditions are given, such as Poincar\'e duality. Those results give a complete description of arrangements (with geometric lattice and with the codimension condition on the subspaces) such that $M(\Ar)$ is rationally elliptic, and show that most arrangements have an hyperbolic complement.
\end{abstract}

\maketitle

\section{Introduction}
%---------------------------------------------------------------
Let $l$ be an integer. A subspace arrangement $\Ar$ is a finite set of affine subspaces in $\C^l$. We say that $\Ar = \{x_1, \dotsc, x_n\}$ is central if all the affine subspaces $x_i$ are vector subspaces.

To every arrangement, we associate the set of non empty intersections of elements of $\Ar$. This set $L(\Ar)$ is partially ordered by $x \leq y \iff y \subseteq x$. Let $x, y \in L(\Ar)$. If $\Ar$ is central, we can define two operations on $L(\Ar)$~: the meet $x \wedge y = \cap\{z \in L(\Ar) \st x \cup y \subset z\}$ and the join $x \vee y = x \cap y$. With these two operations, $L(\Ar)$ is a lattice.

For every $x \in L(\Ar)$, there exists a longest maximal chain $\C^l < x_1 < \dots < x_r = x$. We say that the rank of $x$, $\rk(x)$, is $r$.  The lattice $L(\Ar)$ is called geometric if, for every $x, y \in L(\Ar)$, we have~: $\rk(x) + \rk(y) \geq \rk(x \wedge y) + \rk(x \vee y)$. % An arrangement $\Ar$ is called geometric if it is central and its lattice $L(\Ar)$ is geometric. 

The complement of a subspace arrangement $\Ar$ is the topological space \begin{equation*}M(\Ar) = \C^l \setminus \bigcup \Ar. \end{equation*}
In 2002, S. Yuzvinsky described a rational model for $M(\Ar)$ in \cite{yu02}. Later on, in \cite{yu05}, S. Yuzvinsky and E. Feichtner proved that if the lattice $L(\Ar)$ is geometric, then $M(\Ar)$ is formal and they give a simpler differential graded algebra model for $M(\Ar)$.

In section~\ref{sec:rht}, we state some basic properties of rational homotopy theory. The rational model defined by Yuzvinsky for $M(\Ar)$ is recalled in section~\ref{sec:rmsa}.  Arrangements with Poincar\'e duality are studied in section~\ref{sec:pd}. Finally, the main results are contained in section~\ref{sec:proof}.  

Briefly, the theorem~\ref{theo:a} shows that, under some conditions, the following statements are equivalent~: the subspace arrangement $\Ar$ has a rationally elliptic complement $M(\Ar)$, $\codim \cap_{x \in \Ar} x = \sum \codim x$, $M(\Ar)$ has Poincar\'e duality, $M(\Ar)$ has the homotopy type of a product of odd dimensional spheres. The theorem~\ref{theo:b} gives a geometric interpretation~: these statements are equivalent to the fact that $x_1^\perp + \dotso + x_n^\perp$ is a direct sum.  So, every arrangement with a geometric lattice and rationally elliptic complement is obtained by taking a direct sum $y_1 \oplus \dotso \oplus y_n$ of vector subspaces in $\C^l$ and then taking their orthogonal complement~: $\Ar = \{y_1^\perp, \dotsc, y_n^\perp\}$. 

It shows that most arrangements have an hyperbolic complement.  In that case, which is easy to check with the condition $\codim \cap_{x\in \Ar} x \neq \sum \codim(x)$, the sequence $\sum_{i \leq p} \rk \pi_i(M\Ar))$ has an exponential growth and for any integer $N$, there are infinitely many $q$ with $\rk \pi_q(M(\Ar)) \geq N$.

I would like to thank the referee for his/her work. In particular, the comments about the geometric interpretation were very helpful.

\section{Rational homotopy theory} \label{sec:rht}
%---------------------------------------------------------------
For the basic facts on rational homotopy, we will refer to the classical references (see \cite{su77} or \cite{fe00}).

Let $V$ be a graded vector space. The free commutative algebra on $V$, $\Lambda V$, is by definition the tensor product of the symmetric algebra on $V^{\text{even}}$ by the exterior algebra on $V^{\text{odd}}$. A minimal model is a differential graded algebra of the form $(\Lambda V, d)$ where $d(V) \subset \Lambda^{\geq 2} V$, and such that there is a basis of $V$, $(x_a)_{a \in A}$, indexed by a well-ordered set with the property that $d(x_a) \in \Lambda(x_b)_{b < a}$.

Each 1-connected space $X$ with finite Betti numbers admits a minimal model $(\Lambda V, d)$ that is unique up to isomorphism and that contains all the rational homotopy type of $X$. In particular, $\dim V^n = \dim \pi_n(X) \otimes \Q$. 

\begin{definition}
The space $X$ is called \emph{formal} if there is a quasi-isomorphism $(\Lambda V, d) \to (H^\star(X, \Q), 0)$.
\end{definition}

In the case of subspace arrangements, it is known that if the lattice $L(\Ar)$ is geometric then the space $M(\Ar)$ is formal (see~\cite{yu05}).

The dichotomy theorem in rational homotopy theory states that finite 1-connected CW-complexes are either elliptic or hyperbolic, with the following properties~: if $X$ is elliptic, then $\pi_n(X) =0$ for $n$ large enough and $H^\star(X;\Q)$ satisfies Poincar\'e duality.  If $X$ is hyperbolic, then the sequence $\dim \pi_n(X)$ has an exponential growth.

We will use the following theorem (see~\cite{fe82} for details)
\begin{theorem}[F\'elix-Halperin] \label{theo:fh}
If a space $X$ is elliptic and formal, then its minimal model has the form $(\Lambda V, d) = (\Lambda V_0 \oplus V_1, d)$ with $V_0 = V_0^{\text{odd}}\oplus V_0^{\text{even}}$, $\dim V_0^{\text{even}} = \dim V_1$, $V_1 = V_1^{\text{odd}}$, $dV_0 = 0$ and $dV_1 \subset \Lambda V_0$. Moreover, the injection $(\Lambda V_0^{\text{odd}}, 0) \to (\Lambda V, d)$ induces an injective map in cohomology.
\end{theorem}

Finally, to use rational homotopy theory, we need spaces that are 1-connected.  The following lemmas will show that the space $M(\Ar)$ is 1-connected if the subspaces have all a codimension $\geq 2$.
\begin{lemma} \label{lemm:ext}
Let $\Ar = \{x_1, \dotsc, x_q\}$ be a central arrangement in $\C^l$ such that $\codim x_i \geq 2$. Let $y_i = x_i \cap S^{2l-1}$ and $f\de S^1 \to S^{2l-1} \setminus \cup_{i=1}^q y_i$ be a smooth map. Then $f$ extends to a map $\bar{f} \de D^2 \to S^{2l-1} \setminus \cup_{i=1}^q y_i$~:
\begin{equation*} \xymatrix@C=4mm@R=3mm{S^1 \ar[rr]^{f} \ar@{^{(}->}[dr] && S^{2l-1} \setminus \cup_{i=1}^q y_i \\  & D^2 \ar[ur]_{\bar{f}}}
\end{equation*}
\end{lemma}
\begin{proof}
The proof is done by induction on $q$.  The case $q=1$ is a direct consequence of corollary~15.7 in \cite{br93}. Let's assume the result true until $q-1$. Let $f\de S^1 \to S^{2l-1} \setminus \cup_{i=1}^q y_i$. By induction, we know that there exists a $\tilde{f}\de D^2 \to S^{2l-1} \setminus \cup_{i=1}^{q-1} y_i$ such that $\tilde{f}|_{S^1} = f$. Let $r >0$ such that $r < \dist(\tilde{f}(D^2), \cup_{i=1}^{q-1} y_i)$ and
\begin{equation*}
T = \{z \in S^{2l-1} \st \dist(z, \cup_{i=1}^{q-1} y_i) < r\}.
\end{equation*}
The way we constructed $T$ implies that $\im \tilde{f} \subset S^{2l-1} \setminus T$, which is a $(2l-1)$-dimensional manifold with boundary.  In it, $y_q \setminus T$ is a compact submanifold of dimension $< 2l-4$ (because it is of codimension $\geq 2$ in $\C^l$). The corollary~15.6 in~\cite{br93} gives the existence of a smooth map $\bar{f} \de D^2 \to S^{2l-1} \setminus T$ such that~: $\bar{f}|_{S^1} = \tilde{f}|_{S^1} = f$ and $\bar{f}(D^2)$ is transverse to $y_q \setminus T$. But $\dim \bar{f}(D^2) + \dim(y_q \setminus T) \leq 2 + 2l-4 < 2l-1$. So, transversality can only happen if $\bar{f}(D^2) \cap (y_q \setminus T) = \emptyset$.  Therefore, the application $\bar{f}(D^2)$ is such that $\im\bar{f}(D^2) \subset (S^{2l-1} \setminus T) \setminus (y_q \setminus T) \subset S^{2l-1} \setminus \cup_{i=1}^q y_i$. 
\end{proof}
\begin{lemma} \label{lemm:muc}
Let $\Ar$ be a subspace arrangement such that for each $x \in \Ar$, $\codim x \geq 2$. Then the space $M(\Ar)$ is 1-connected.
\end{lemma}
\begin{proof}
Let $f\de S^1 \to M(\Ar)$ be a map. Since the $x_i$ are vector spaces, we can define the homotopy $h_t = (1-t)f + t\frac{f}{||f||}$. We can assume that the map $h_1 \de S^1 \to S^{2l-1} \setminus(\cup_{x \in \Ar} (x \cap S^{2l-1})$ is smooth. So, lemma~\ref{lemm:ext} can be applied and shows that $f \simeq h_1 \simeq \star$.  It implies that $\pi_1(M(\Ar)) = 0$, so $M(\Ar)$ is 1-connected.
\end{proof}

\section{Rational model of subspace arrangements} \label{sec:rmsa}
%---------------------------------------------------------------
Let $\Ar$ be a central arrangement of subspaces in $\C^l$. Yuzvinsky defined the relative atomic differential graded algebra $D_\Ar = (D, d)$ associated with an arrangement as follows (see~\cite{yu05})~: choose a linear order on $\Ar$. The chain complex $(D, d)$ is generated by all subsets $\sigma \subseteq \Ar$.  For $\sigma = \{x_1, \dotsc, x_n\}$, we define the differential by
\begin{equation*}
d\sigma = \sum_{j:\vee(\sigma \setminus \{x_j\}) = \vee \sigma} (-1)^j(\sigma\setminus\{x_j\})
\end{equation*}
where the indexing of the elements in $\sigma$ follows the linear order imposed on $\Ar$.  With $\deg(\sigma) = 2 \codim \vee \sigma - |\sigma|$, $(D,d)$ is a cochain complex.  Finally, we need a multiplication on $(D,d)$.  For $\sigma, \tau \subseteq \Ar$, 
\begin{equation*}
\sigma \cdot \tau = \left\{\begin{aligned} (-1)^{\sgn\epsilon(\sigma, \tau)} \sigma \cup \tau &\text{ if } \codim \vee \sigma + \codim \vee \tau = \codim \vee(\sigma \cup \tau) \\ 0 & \text{ otherwise} \end{aligned}\right.
\end{equation*}
where $\epsilon(\sigma, \tau)$ is the permutation that, applied to $\sigma \cup \tau$ with the induced linear order, places elements of $\tau$ after elements of $\sigma$, both in the induced linear order.

A subset $\sigma \subseteq \Ar$ is said to be independant if $\rk( \vee \sigma) = |\sigma|$. When $\Ar$ is an arrangement with a geometric lattice, we have the following property~: $H^\star(M(\Ar))$ is generated by the classes $[\sigma]$, with $\sigma$ independant (\cite{yu05}).

\section{Poincar\'e duality} \label{sec:pd}
%---------------------------------------------------------------
Having Poincar\'e duality is a strong statement for a subspace 
arrangement with geometric lattice. That condition alone determines the minimal model of the complement.

For $\Ar = \{x_1, \dotsc, x_n\}$ a subspace arrangement with $L(\Ar)$ geometric, let $M_r$ be the greatest element in $L(\Ar)$ and $\C^l < M_1 < \dotso < M_{r-1} < M_r$ be any maximal chain in $L(\Ar)$.  In particular, $\rk(M_i)=i$ for $1\leq i \leq r$. Let $X_i = \{x \in \Ar \st x < M_i \}$. We can construct a chain complex $C^i_\star$ : $C^i_p$ is the module generated by all the linear combinations of the $\sigma \subset \Ar$ such that $\vee \sigma = M_i$ and $|\sigma| = p$.  With the differential defined in the Yuzvinsky model of $M(\Ar)$ and $\deg(\sigma) = |\sigma|$, $C^i_\star$ is clearly a chain complex. 
\begin{lemma} \label{lemm:pda}
Let $\Ar = \{x_1, \dotsc, x_n\}$ be a subspace arrangement with geometric lattice.  If $M(\Ar)$ has Poincar\'e duality and $1 \leq k \leq r$, then $\dim H_k(C^k_\star) = 1$.
\end{lemma}
\begin{proof}
In this proof, $X_0$ is the empty set. Let $E^i_p$ ($1 \leq i \leq r$) be the submodule of $C^i_p$ generated by 
\begin{itemize}
\item all the $\sigma \subset \Ar$ such that $|\sigma| = p$, $\vee \sigma = M_i$ and $\sigma$ contains at least 2 elements of $X_i \setminus X_{i-1}$, 
\item all the elements $\{x_{i_1}, x_{i_2}, \dotsc, x_{i_{p-1}}, y_1\} - \{x_{i_1}, \dotsc, x_{i_{p-1}}, y_2\}$ with $x_{i_j} \in X_{i-1}$ and $y_1, y_2 \in X_{i} \setminus X_{i-1}$.
\end{itemize}
It is easy to check that $E^i_\star$ is a subcomplex of $C^i_\star$. We have the following short exact sequences~:
\begin{equation*}
0 \to E^i_\star \to C^i_\star \to C^i_\star/ E^i_\star \to 0 .
\end{equation*}
Since $\Ar$ has a geometric lattice, the cohomology is generated by the classes $\sigma \in \Ar$ such that $\sigma$ is independant.  So, the (reduced) homology of $C^i_\star$ is $0$ in every degree except possibly the $i^\text{th}$ degree.  We have~: $H_\star(C^i_\star) = H_{i}(C^i_\star)$. For $1 \leq i \leq r$, we know that $\rk(M_i) = i$. It means that $E^{i+1}_{i}$ is an empty set and $H_{i}(E^{i+1}_\star) = 0$. Hence, the long exact sequence in homology associated with the short exact sequence above implies that the map $H_{i+1}(C^{i+1}_\star) \to H_{i+1}(C^{i+1}_\star/E^{i+1}_\star)$ is surjective.

Since $X_{i+1} \setminus X_{i}$ is non empty (because $\vee X_{i+1} = M_{i+1}$ and $\vee X_i = M_i)$, we can fix some $y \in X_{i+1} \setminus X_{i}$.  This $y$ define maps $\varphi_p \de C^i_p \to C^{i+1}_{p+1}/E^{i+1}_{p+1}$ sending $\{x_1, \dotsc, x_p\}$ to $[\{x_1, \dotsc, x_p, y\}]$. Since the lattice is geometric, if $\vee \{x_1,  \dotsc, \hat{x_j}, \dots, x_p\} < M_i$, then $\vee \{x_1, \dotsc, \hat{x_j}, \dotsc, x_p, y\} < M_{i+1}$. Therefore, the maps $\varphi_p$ commute with the differentials and define an isomorphism of chain complex $(C^i_\star)_p \to (C^{i+1}_\star / E^{i+1}_\star)_{p+1}$.  Hence, $H_p(C^i_\star) = H_{p+1}(C^{i+1}_\star/ E^{i+1}_\star)$. But, we proved that the map $H_{i+1}(C^{i+1}_\star) \to H_{i+1}(C^{i+1}_\star/ E^{i+1}_\star)$ is a surjection. So\begin{equation*}
\dim H_{i}(C^i_\star) = \dim H_{i+1}(C^{i+1}_\star/ E^{i+1}_\star) \leq \dim H_{i+1}(C^{i+1}_\star).
\end{equation*}
Since the space $M(\Ar)$ has Poincar\'e duality, there is a unique cohomology class in the highest degree in $H^\star(D_\Ar)$.  That cohomology class is represented by an independant $\sigma \in \Ar$ such that $|\sigma| = r$. If $\dim H_r(C_\star^r) \geq 2$, then there is another class $[\tau]$ and by Poincar\'e duality, there is an element $[\rho]$ in $H^\star(D_\Ar)$ such that $[\sigma] = [\rho] [\tau]$, but this is impossible by the multiplication law because $\codim \vee \sigma= \codim \vee \tau$. Therefore, $\dim H_r(C_\star^r) = 1$. From the following sequence of inequalities
\begin{equation*}
1 = \dim H_1(C^1_\star) \leq \dim H_2(C^2_\star) \leq \dots \leq \dim H_r(C^r_\star) = 1.
\end{equation*}
we deduce that $\dim H_k(C^k_\star) = 1$ for all $1 \leq k \leq r$.
\end{proof}

\begin{lemma} \label{lemm:pdb}
Let $\Ar = \{x_1, \dotsc, x_n\}$ be a subspace arrangement such that $L(\Ar)$ is geometric and $M(\Ar)$ has Poincar\'e duality. Let $M \in L(\Ar)$ with $\rk(M) = i$ and let $X(M) = \{x \in \Ar \st x \leq M \}$, then $\# X(M) = i$.
\end{lemma}
\begin{proof}
We prove the result by induction on $\rk M$. It is clear for $i = 1$. Now, let us suppose that it is true for all $N \in L(\Ar)$ with $\rk N \leq i-1$ and let $M \in L(\Ar)$ with $\rk M = i$.  Denote by $M_1 < M_2 < \dotso < M_i = M < M_{i+1} < \dotso < M_r$ a maximal sequence in $L(\Ar)$, and write
\begin{equation*}
X(M_{i-1}) = \{x_1, \dotsc, x_{i-1}\} \quad \text{and} \quad
X(M) = \{x_1, \dotsc, x_{i-1}, x_{i-1+1}, \dotsc, x_{i-1+l}\}.
\end{equation*}
We consider the chain complex $C_\star^i$ defined in lemma~5 for that maximal chain.  Remark first that if $\{ x_{n_1}, \dotsc, x_{n_{i+1}}\} \subset X(M)$ with $\vee x_{n_i} = M$, then for each $k$, $\vee \{ x_{n_1}, \dotsc, \hat{x}_{n_k}, \dotsc, x_{n_{i+1}}\} = M$, because otherwise $\vee \{ x_{n_1}, \dotsc, \hat{x}_{n_k}, \dotsc, x_{n_{i+1}}\}$ is an element $N$ in $L(\Ar)$ with $\rk N < i$.  So, there are $i$ subspaces $y_j$ with $y_j < N$, in contradiction with our induction hypothesis. Therefore, if $\{x_{n_1}, \dotsc, x_{n_{i+1}}\} \subset C_{i+1}^i$, then 
\begin{equation*}
d\{x_{n_1}, \dotsc, x_{n_{i+1}}\} = \sum_{j=1}^{i+1} (-1)^j \{x_{n_1}, \dotsc, \hat{x}_{n_j}, \dotsc, x_{n_{i+1}}\}.
\end{equation*}
In the complex $C_\star^i$, every cycle of degree $i$ is equivalent to a sum $\sum \alpha_j \{x_1, x_{j_2}, \dotsc, x_{j_i}\}$.  Indeed, if $1 \not\in \{j_1, \dotsc, j_i\}$, then
\begin{equation*}
\{x_{j_1}, \dotsc, x_{j_i}\} = -d\{x_1, x_{j_1}, \dotsc, x_{j_i}\} + \sum_{k=2}^i (-1)^k\{x_1, \dotsc, \hat{x}_{j_k}, \dotsc, x_{j_i}\}.
\end{equation*}
Now, no cycle of the form $\sum_j \alpha_j\{x_1, x_{j_2}, \dotsc, x_{j_i}\}$ is a boundary.  Suppose this is the case, we have : 
\begin{equation*}
\sum\nolimits_j \alpha_j\{x_1, x_{j_2}, \dotsc, x_{j_i}\} = d\left[ \sum\nolimits_m \beta_m \{x_1, x_{m_1}, \dotsc, x_{m_i}\} + \sum\nolimits_n \gamma_n \{x_{n_1}, \dotsc, x_{n_{i+1}}\}\right]
\end{equation*}
with $1 \not\in \{x_1, \dotsc, n_{i+1}\}$.  Developing the differential, we get
\begin{equation*}
0 = -\sum\nolimits_m \beta_m\{x_{m_1}, \dotsc, x_{m_i}\} + \sum\nolimits_n \gamma_n \left(\sum_{k=1}^{i+1} (-1)^k \{x_{n_1}, \dotsc, \hat{x}_{n_k}, \dotsc, x_{n_{i+1}}\} \right).
\end{equation*}
We deduce that
\begin{multline*}
d\left(\sum\nolimits_n \gamma_n\{x_1, x_{n_1}, \dotsc, x_{n_{i+1}} \}\right) \\
= -\sum_n \gamma_n \{x_{n_1}, \dotsc, x_{n_{i+1}} \} - \sum_n \gamma_n\left( \sum_{k=1}^{i+1} (-1)^k \{x_1, x_{n_1}, \dotsc, \hat{x}_{n_k}, \dotsc, x_{n_{i+1}}\}\right) \\
= - \sum\nolimits_n \gamma_n \{x_{n_1}, \dotsc, x_{n_{i+1}} \} -\sum\nolimits_m \beta_m\{x_1, x_{m_1}, \dotsc, x_{m_i}\}.
\end{multline*}
Since $d^2=0$, this gives $\sum_j \alpha_j\{x_1, x_{j_2}, \dotsc, x_{j_i} \} = 0$.  

We deduce from the above calculation that $l=1$, i.e. $X(M) = \{x_1, \dotsc, x_i\}$, because otherwise, the cycles $\{x_1, \dotsc, x_i\}$ and $\{x_1, \dotsc, x_{i-1}, x_{i+1}\}$ would be linearly independant in homology, in contradiction with lemma~\ref{lemm:pda}.
\end{proof}

\begin{proposition} \label{prop:pd}
Let $\Ar = \{x_1, \dotsc, x_n\}$ be a subspace arrangement.  If $L(\Ar)$ is geometric and $M(\Ar)$ has Poincar\'e duality, then the minimal model of $M(\Ar)$ is the algebra $(\Lambda(y_1, \dotsc, y_n), 0)$ where $\deg y_i = 2 \codim x_i -1$.
\end{proposition}
\begin{proof}
It is a consequence from lemma~\ref{lemm:pdb}.  This lemma shows that every subset $\sigma \subset \Ar$ are independant. Therefore, any product $\{x_{i_1}\} \cdot \dotso \cdot \{x_{i_l}\} \neq 0$, and all the products are different in cohomology. 
\end{proof}

\section{Main results} \label{sec:proof}
%---------------------------------------------------------------
Now, everything is in place to prove the main results.  The first theorem uses rational homotopy theory and gives some equivalent conditions to the fact that $M(\Ar)$ is rationally elliptic.  With some linear algebra, the second theorem shows that the condition (3) has a geometric interpretation in term of the orthogonal subspaces $x_i^\perp$.
\begin{theorem}  \label{theo:a}
Let $\Ar$ be a subspace arrangement with a geometric lattice such that every $x \in \Ar$ has $\codim(x) \geq 2$. Then the following conditions are equivalent~:
\begin{enumerate}
\item $M(\Ar)$ is rationally elliptic,
\item $M(\Ar)$ has the rational homotopy type of a product of odd dimensional spheres,
\item $\codim \cap_{x \in \Ar} x = \sum \codim x$,
\item $M(\Ar)$ has the homotopy type of a product of odd dimensional spheres,
\item $M(\Ar)$ has Poincar\'e duality.
\end{enumerate}
\end{theorem}
\begin{proof}
\emph{(1) implies (2).} Since $L(\Ar)$ is geometric, we know (see~\cite{yu05}) that $M(\Ar)$ is a formal space. If $M(\Ar)$ is elliptic, we can apply theorem~\ref{theo:fh}. By definition of the differential, every $x \in \Ar$, $\{x\}$ is a generator in cohomology for the rational model described in section~\ref{sec:rmsa}. The degree of $[\{x\}]$ is $2 \codim x - 1$.  Therefore, the $(\{x\})_{x \in \Ar}$ form a linearly independant sequence in $V_0^\text{odd}$. By theorem~\ref{theo:fh}, we have an injective map
\begin{equation*}
\rho\de \Lambda_{x \in \Ar} [\{x\}] \to H^\star(\Lambda V).
\end{equation*}
In particular, for each sequence $x_1, \dotsc, x_n$ in $\Ar$, with $
x_i \neq x_j$, we have 
\begin{equation*}\{x_1 \} \cdot \{x_2\} \cdot \dotsc \cdot \{x_n\} \neq 0
\end{equation*}
because their product is non zero in cohomology.  Therefore, we have the following equality $\prod_{i=1}^n \{x_i\} = \pm \{x_1, x_2, \dotsc, x_n\}$ and $[\{x_1, \dotsc, x_n\}] \neq 0$ (in cohomology).

The map $\rho$ is surjective because, for each independant set $\{x_1, \dotsc, x_n\}$ (which generates $H^\star(M(\Ar))$), we have $[\{x_1, \dotsc, x_n\}] = \pm \prod_{i=1}^n [\{x_i\}]$, which is in the image of $\rho$.  It implies that the map $\rho$ is an isomorphism. By lemma~\ref{lemm:muc}, $M(\Ar)$ is 1-connected. Therefore,  $M(\Ar)$ has the rational homotopy type of a product of odd dimensional spheres.

\emph{(2) implies (3).} We showed that the product $\prod_{x \in \Ar} \{x\} \neq 0$.  By definition of the product, it implies that $\codim \cap_{x \in \Ar} x = \sum_{x \in \Ar} \codim x$.

\emph{(3) implies (4).} Let $\Ar = \{x_1, \dotsc, x_n\}$ be a subspace arrangement in $\C^l$ such that $\codim\cap x_i = \sum \codim x_i$.  The quotient map $p \de \C^n \to \C^n/(\cap x_i)$ induces a homotopy equivalence $(\C^n \setminus \cup x_i) \to ((\C^n / \cap x_i) \setminus \cup (x_i/\cap x_i))$.  Hence we can assume that $\cap x_i = 0$.

Let's write $x_i = \ker(H_i \de \C^n \to \C^{n_i})$. The map 
\begin{equation*}
(H_1, H_2, \dotsc, H_n) \de \C^n \to \prod \C^{n_i}
\end{equation*}
is an isomorphism, which induces an homotopy equivalence 
\begin{equation*}
\C^n \setminus \cup x_i \to \prod_{i=1}^n (\C^{n_i} \setminus \{0\}).
\end{equation*}
But the injective map $\prod(S^{2n_i-1}) \to \prod(\C^{n_i} \setminus \{0\})$ is an homotopy equivalence.  Therefore $M(\Ar) = \C^n \setminus \cup x_i$ has the homotopy type of a product of odd dimensional spheres.

\emph{(4) implies (1).} Obvious.

\emph{(1) implies (5).} Obvious.

\emph{(5) implies (2).} Direct consequence from proposition~\ref{prop:pd}.
\end{proof}

\begin{theorem} \label{theo:b}
Let $\Ar = \{x_1, \dotsc, x_n\}$ be a subspace arrangement. Then the following conditions are equivalent~:
\begin{enumerate}
\item $\codim \cap_{i=1}^n x_i = \sum_{i=1}^n \codim x_i$,
\item the sum $x_1^\perp + \dotso + x_n^\perp$ is a direct sum.
\end{enumerate}
\end{theorem}
\begin{proof}
First, let's prove by induction on $k$, $2 \leq k \leq n$, that~: \[
\left(\cap_{j=1}^k x_j\right)^\perp = \sum\nolimits_{i=1}^k x_i^\perp.
\]
For $k=2$, it gives $(x_1 \cap x_2)^\perp = x_1^\perp + x_2^\perp$, which is a well-known fact. Now, let's suppose that the formula is true until $k-1$.  We have~: \[
\left(\cap_{j=1}^k x_j\right)^\perp = (\cap_{j=1}^{k-1} x_j \cap x_k)^\perp = (\cap_{j=1}^{k-1} x_j )^\perp + x_k^\perp.
\]
Using the induction hypothesis concludes the proof. Now, we can prove the theorem~:
\begin{multline*}
x_1^\perp + \dotso + x_n^\perp \text{ is a direct sum} \iff \sum\nolimits_{i=1}^n \dim x_i^\perp = \dim\left(\sum\nolimits_{i=1}^n x_i^\perp\right) \\
 \iff \sum\nolimits_{i=1}^n \codim x_i = \dim \left( \cap_{i=1}^n x_i  \right)^\perp  \iff \sum\nolimits_{i=1}^n \codim x_i = \codim \cap_{i=1}^n x_i. \qedhere
\end{multline*}
\end{proof}


\begin{thebibliography}{xx}
\bibitem{br93} \textsc{G. Bredon}, \textit{Topology and Geometry}, Graduate texts in Mathematics, Springer Verlag, 1993.
\bibitem{fe82} \textsc{Y. F\'elix, S. Halperin}, \textit{Formal spaces with finite-dimensional rational homotopy}, Trans. Amer. Math. Soc. 270 (1982), no. 2, 575--588.
\bibitem{fe00} \textsc{Y. F\'elix, S. Halperin, J.-C. Thomas}, \textit{Rational Homotopy Theory}, Springer-Verlag, 2000.
\bibitem{su77} \textsc{D. Sullivan}, \textit{Infinitesimal computations in topology}, Publ. IHES 47 (1977), 267--331.
\bibitem{yu02} \textsc{S. Yuzvinsky}, \textit{Small rational model of subspace complement}, Trans. Amer. Math. Soc. 354 (2002), no. 5, 1921--1945.
\bibitem{yu05} \textsc{E. Feichtner, S. Yuzvinsky}, \textit{Formality of the complements of subspace arrangements with geometric lattices}, \url{arXiv:math.AT/0504321}, 2005.
\end{thebibliography}
\end{document}